\documentclass{amsart}
\usepackage{amscd,amssymb,amsopn,amsmath,amsthm,graphics,amsfonts,accents,enumerate,verbatim,calc}
\usepackage[dvips]{graphicx}
\usepackage{pgfplots}
\usepackage[colorlinks=true,linkcolor=red,citecolor=blue]{hyperref}
\usepackage[all]{xy}
\usepackage{cite}
\usepackage{tikz}
\usepackage{tikz-cd}
\usepackage{mathrsfs}
\usepackage[colorinlistoftodos]{todonotes}
\usepackage{calrsfs}

\calclayout

\newcommand{\rt}{\rightarrow}

\newcommand{\SA}{\mathscr{A}}

\newcommand{\SI}{\mathscr{I}}

\newcommand{\SX}{\mathscr{X}}

\newcommand{\CA}{\mathcal{A} }

\newcommand{\CC}{\mathcal{C} }

\newcommand{\CF}{\mathcal{F} }

\newcommand{\CI}{\mathcal{I} }
\newcommand{\CJ}{\mathcal{J} }

\newcommand{\Mod}{{\rm{Mod\mbox{-}}}}

\newcommand{\Flat}{{\rm{Flat}}}

\newcommand{\OC}{\mathbb{OC}}

\newcommand{\im}{{\rm{Im}}}
\newcommand{\op}{{\rm{op}}}

\newcommand{\PE}{{\rm{PE}}}

\newcommand{\Coker}{{\rm{Coker}}}
\newcommand{\Ker}{{\rm{Ker}}}

\newcommand{\Prj}{{\rm{Prj}}}
\newcommand{\PPrj}{{\rm{PPrj}}}
\newcommand{\PInj}{{\rm{PInj}}}

\newcommand{\SWC}{{S\mbox{-}\rm{WC}}}
\newcommand{\SSF}{{S\mbox{-}\rm{SF}}}
\newcommand{\SPh}{{S\mbox{-}{\rm{Ph}}}}
\newcommand{\SPI}{{S\mbox{-}{\rm{PInj}}}}
\newcommand{\SPE}{{S{\rm{PE}}}}
\newcommand{\Ph}{{\rm{Ph}}}
\newcommand{\bS}{\mathbb{S}}

\newcommand{\Tor}{{\rm{Tor}}}
\newcommand{\Hom}{{\rm{Hom}}}
\newcommand{\Ext}{{\rm{Ext}}}

\theoremstyle{plain}
\newtheorem{theorem}{Theorem}[section]
\newtheorem{corollary}[theorem]{Corollary}
\newtheorem{lemma}[theorem]{Lemma}
\newtheorem{facts}[theorem]{Facts}
\newtheorem{proposition}[theorem]{Proposition}

\theoremstyle{definition}
\newtheorem{definition}[theorem]{Definition}
\newtheorem{example}[theorem]{Example}

\newtheorem{remark}[theorem]{Remark}

\newtheorem{s}[theorem]{}

\theoremstyle{plain}

\theoremstyle{definition}

\numberwithin{equation}{section}

\begin{document}

\title[Strong purity and phantom morphisms]{Strong Purity and Phantom Morphisms}

\author[Hafezi, Asadollahi, Sadeghi and Zhang]{R. Hafezi, J. Asadollahi, S. Sadeghi and Y. Zhang}

\address{School of Mathematics and Statistics, Nanjing University of Information Science and Technology, Nanjing, Jiangsu 210044, P. R. China}
\email{hafezi@nuist.edu.cn }

\address{Department of Pure Mathematics, Faculty of Mathematics and Statistics, University of Isfahan, P.O.Box: 81746-73441, Isfahan, Iran}
\email{asadollahi@sci.ui.ac.ir, asadollahi@ipm.ir }

\address{School of Mathematics, Institute for Research in Fundamental Sciences (IPM), P.O.Box: 19395-5746, Tehran, Iran}
\email{somayeh.sadeghi@ipm.ir }

\address{School of Mathematics and Statistics, Nanjing University of Information Science and Technology, Nanjing, Jiangsu 210044, P. R. China}
\email{zhangy2016@nuist.edu.cn }

\makeatletter \@namedef{subjclassname@2020}{\textup{2020} Mathematics Subject Classification} \makeatother



\begin{abstract}
Let $R$ be a commutative ring and $S \subseteq R$ be a multiplicative subset. We introduce and study the concept of $S$-purity based on the notion of $S$-strongly flat modules.  The class of $S$-pure injective modules will be studied. We demonstrate that this class is enveloping and explore its closedness under extension. The concept of purity is closely connected to the existence of phantom maps. So we will delve into the study of the $S$-phantom morphisms. We will establish that the $S$-phantom ideal is a precovering ideal and examine the situations where it becomes a covering ideal. Finally, in the last section, we will investigate an ideal version of the `Optimistic Conjecture', raised by Positselski and Sl\'{a}vik.
\end{abstract}

\maketitle

\section{Introduction}
Let $R$ be a commutative ring, $S \subseteq R$ be a multiplicative subset and $R_S$ be the localization of $R$ at $S$. An $R$-module $G$ is called $S$-strongly flat if $\Ext^1_R(G, C)=0$, for every $S$-weakly cotorsion module $C$, i.e. for every $R$-module $C$ with the property that $\Ext^1_R(R_S, C)=0$. Let $\SSF$, resp. $\SWC$, denote the class of $S$-strongly flat, resp. $S$-weakly cotorsion, $R$-modules.

These classes of modules are introduced in \cite[Section 2]{Tr1}, when $R$ is a commutative integral domain and $R_S$ is the quotient field of $R$. In \cite{FS}, they considered an arbitrary commutative ring but $S$ is the multiplicative set of all regular elements. In \cite{P}, the general case where $S$ may contain zero-divisors was considered. For a more recent account on these two classes see \cite{PS1} and \cite{BP}. In the non-commutative case this class is studied in \cite{FN}, when $R$ is a right Ore domain with classical right quotient ring $Q$. See also the survey \cite{Sa}.

It is known that the pair $(\SSF, \SWC)$ forms a complete cotorsion pair. This, in particular, implies that $\SSF$ is a precovering class. So a natural question is when it is a covering class. This question is raised in \cite{Tr2}. In \cite{BS} it is shown that when $R$ is an integral domain and $R_S$ is the quotient field of $R$, then $\SSF$ is a covering class if and only if  $R$ is almost perfect, that is, all proper quotients of $R$ are perfect. Recall that a ring is perfect if all its flat modules are projective. This result is generalized in \cite{FS} to an arbitrary commutative ring $R$ with zero-divisors when $R_S$ is the total ring of quotients. A wide generalization for the case when $S$ may contain zero-divisors is studied in \cite{BP}. It is shown that all $R$-modules have $S$-strongly flat covers if and only if $R$ is an $S$-almost perfect ring, i,e, the localization $R_S$ is a perfect ring and, for every $s \in S$, the quotient ring $R/sR$ is perfect as well. This, in turn, implies that all flat $R$-modules are $S$-strongly flat.

On the other hand, ideal approximation theory is developed in \cite{FGHT}. It provides a wide generalization of the classical approximation theory. In this new theory, modules are replaced by morphisms and additive subcategories are replaced by ideals. The phantom ideal is one of the most important examples of ideals. The concept of purity is closely related to the occurrence of phantom maps. A morphism $\varphi: U\rt Z$ of $R$-modules is called a phantom morphism if the pullback along $\varphi$ of any short exact sequence with the right term $Z$ gives rise to a pure exact sequence. This notion was first studied in homotopy theory \cite{Mc}, then in the setting of triangulated categories \cite{N}, and in the stable category of $k[G]$ modules in \cite{Gn}. It was also studied in the general setting of modules over an associative ring with identity in \cite{H1}. The phantom ideal of the category of $R$-modules, denoted by $\Ph$, consists of all phantom morphisms. In \cite[Theorem 7]{H1}, it is shown that $\Ph$ is a covering ideal.

This paper, aims to develop an ideal approximation theory by introducing the concept of $S$-strong flatness. We are working within a general setting where $R$ is an arbitrary commutative ring and the multiplicative set $S$ may contain zero-divisors. By definition every $S$-strongly flat module $G$ is flat, and thus a short exact sequence $\delta$ ending at an $S$-strongly flat module is pure exact. However, it has additional properties: $\delta_S$, the localization of $\delta$ at $S$, and $\delta/s\delta$ for every $s \in S$ are split exact sequences as sequences of $R_S$-modules and $R/sR$-modules, respectively. This is because $G_S$ and $G/sG$ are projectives as $R_S$ and $R/sR$ modules, respectively (see \cite[Lemma 3.1]{BP}). Building on these properties, we define the notion of $S$-pure exact sequences. A short exact sequence $\delta$ is called $S$-pure exact if it is pure exact and if the sequences $\delta_S=\delta\otimes_RR_S$ and $\delta/s\delta=\delta\otimes_RR/sR$, for every element $s \in S$, are split exact. We show that the collection of all $S$-pure exact sequences is a sub-bifunctor of the bifunctor $\Ext$. Using this, we then define the notion of $S$-pure injective modules as injective modules with respect to the collection of all $S$-pure exact sequences. For a module, we provide an explicit $S$-pure injective preenvelope and, using standard arguments, we show that the class of all $S$-pure injective modules is an enveloping class.

We define $S$-phantom morphisms as those morphisms $\varphi: U\rt Z$ such that the pullback along $\varphi$ of any short exact sequence yields an $S$-pure exact sequence. The collection of all $S$-phantom morphisms form an ideal, denoted by $\SPh$. Based on the existence of $S$-pure injective envelopes, we show that $\SPh$ is a precovering ideal and discuss the situations where it is a covering ideal.

Note that pure-exactness and pure-injectivity for modules over a ring have been introduced by Cohn \cite{Coh}. These notions have been studied in different settings, see for instance \cite{Kra} and \cite{Bel}.

Although the structure of strongly flat modules has been studied for over two decades, their structure is still mysterious. In \cite{PS1} the authors posed a conjecture, called Optimistic Conjecture, hereafter denoted as $(\OC)$, which states that in case the projective dimension of $R_S$ as an $R$-module, does not exceed $1$, a flat $R$-module $F$ is $S$-strongly flat if the localization module $S^{-1}F=F_S$ is a projective $R_S$-module and, for every $s\in S$, the quotient module $F/sF$ is a projective $R/sR$-module. The authors also introduced the concept of right $1$-obtainability of an $R$-module from a given class of $R$-modules, called a `seed class', and provided some positive results towards the conjecture in Theorems 1.3, 1.4, and 1.5 of \cite{PS1}. They demonstrated that the conjecture holds if the multiplicative subset $S$ of $R$ is either countable, consists of nonzero-divisors in $R$, or is bounded, meaning there exists an element $s_0\in S$ such that $sr=0$ for $s\in S$ and $r \in R$ implies that $s_0r=0$. Additionally, in \cite[Proposition 7.13]{BP}, it is shown that the conjecture holds if $R$ is an $S$-h-nil ring, meaning for every element $s \in S$, the ring $R/sR$ is semilocal of Krull dimension $0$. Towards the end of the paper, an ideal version of this conjecture is stated and proven. We remark that domains $R$ with a ring of quotients having projective dimension at most one are extensively studied in \cite{AHT}.

The paper is structured as follows. Section \ref{Sec: Preliminaries}  is dedicated to preliminary results, where we gather known facts about strongly flat modules in the first subsection. For a class $\CC$ of $R$-modules, in \cite{BCE} the authors studied $\CC$-periodic modules, i.e. modules $M$ that fit into a short exact sequence
\[0 \rt M \rt C \rt M \rt 0,\] where $C$ belongs to $\CC$, see also \cite{BG} and \cite{Si}. A $\CC$-periodic module $M$ is called trivial if it belongs to $\CC$. In Subsection \ref{Subsec: PeriodicSF}, we show that $\SSF$-periodic modules are trivial, provided $\OC$ holds for the pair $(R,S)$ (Convention \ref{Conv-2}). In Subsection \ref{Subsec: Phantom}, the basics of ideal approximation theory are reviewed. In particular, we prove some results emphasizing that phantom morphisms could be considered as the morphism counterpart of flat modules.

In Section \ref{Sec: S-Purity} we introduce and study the notion of $S$-purity (Definition \ref{Def; S-pure}). Using the concept of $S$-pure exactness, the notion of $S$-pure injective modules will be defined. We investigate the properties of the class of $S$-pure injective modules. In particular, it is shown that it is a covering class  (Proposition \ref{SsI is Enveloping}). We also consider the situation where this class is closed under extension, see Subsection \ref{Subsec: ClosednessUnder Ext}. It will be shown that if $(\OC)$ holds for the pair $(R, S)$, then the class of $S$-pure injective modules is closed under extension if and only if every $S$-weakly cotorsion module is $S$-pure injective.

Section \ref{Sec: S-Phantom Morphisms} contains the main results of this paper. We introduce the notion of $S$-phantom morphisms (Definition \ref{Def: S-Phantom}) and, using the fact that $S$-pure injective envelopes exist, show that the $S$-phantom ideal is a (special) precovering ideal  (Theorem \ref{Th: S-Phantom Precover}). Moreover, in Subsection \ref{Subsec: S-Phantom Cover}, we investigate when it is a covering ideal. To this end, first, we show that if every phantom morphism is $S$-phantom, then the ring $R$ is $S$-almost perfect, i.e. the localization ring $R_S$ is perfect, and for every $s\in S$, the quotient ring $R/sR$ is perfect. The converse holds if condition $(\OC)$ is satisfied for the pair $(R, S)$.

The final section of the paper discusses an ideal version of the Optimistic Conjecture. It demonstrates that if $\varphi:U\rightarrow Z$ is an $S$-phantom morphism, then $\varphi_S: U_S\rightarrow Z_S$ is a projective morphism, and for every $s\in S$, $\varphi/s\varphi: U/sU\rightarrow Z/sZ$ is a projective morphism. Conversely, it is proven that if $\varphi: U\rightarrow Z$ is a phantom morphism, where $Z$ is $S$-almost flat, then $\varphi$ is $S$-phantom if $\varphi_S: U_S\rightarrow Z_S$ is a projective morphism, and for every element $s \in S$, $\varphi/s\varphi: U/sU\rightarrow Z/sZ$ is a projective morphism (see Propositions \ref{Prop: OC1} and \ref{Prop: OC2}).

\s {\sc{Notations and Conventions.}}\label{Conv-1}
Throughout, $R$ is a commutative ring with identity, $S \subset R$  is a multiplicative subset which may contain some zero-divisors and $R_S$ denotes the localization of $R$ with respect to $S$. We fix a ring $R$. Hence, throughout the paper, $\Prj$ is the class of projective $R$-modules and $\Flat$ denotes the class of flat $R$-modules. The class of $S$-strongly flat modules will be denoted by $\SSF$ and the class of $S$-weakly cotorsion $R$-modules will be denoted by $\SWC$. Let $M$ be an $R$-module. Unless otherwise specified, $M_S$, resp. $M/sM$, for $s \in S$, means $M_S$ as $R_S$-module, resp. $M/sM$ as $R/sR$-module. For instance, when we say $M_S$ is flat, it means $M_S$ is flat as an $R_S$-module.

\section{Preliminaries and some facts}\label{Sec: Preliminaries}
Recall that a ring $R$ is perfect if every flat $R$-module is projective. We need the following definition from \cite{BP}.

\begin{definition}
Let $R$ be a commutative ring and $S \subseteq R$ be a multiplicative subset. We say that $R$ is an $S$-almost perfect ring if the localization ring $R_S$ is perfect and for every $s\in S$, the quotient ring $R/sR$ is perfect.
\end{definition}

\subsection{Strongly flat modules}\label{Subsec: StronglyFlat}
An $R$-module $C$ is called $S$-weakly cotorsion if $\Ext^1_R(R_S, C)=0$. Let $\SWC$ denote the class of $S$-weakly cotorsion $R$-modules. An $R$-module $G$ is called $S$-strongly flat if $G \in {}^{\perp}\SWC$, where \[{}^{\perp}\SWC= \{M \in \Mod R \ | \ \Ext^1_R(M,C)=0, \ {\rm for \ all} \ C \in \SWC \}.\]

It follows from the definition that every cotorsion module is $S$-weakly cotorsion and every $S$-strongly flat module is flat. Recall that an $R$-module $C$ is called cotorsion if $\Ext^1_R(F, C)=0$, for every flat $R$-module $F$.

\begin{facts}\label{Facts} Here we list some known facts on these two classes of modules. See, for instance, \cite{BP}.
\begin{itemize}
\item [$(i)$] Let $G$ be an $R$-module. Then $G$ is $S$-strongly flat if and only if $G$ is a direct summand of an $R$-module $G'$ that fits into the short exact sequence \[0 \rt R^{(\beta)} \rt G' \rt R_{S}^{(\gamma)} \rt 0,\] for some cardinals $\beta$ and $\gamma$.
\item [$(ii)$] If $G$ is an $S$-strongly flat $R$-module, then the localization $G_S$ of $G$ at $S$ is projective as $R_S$-module and, for every $s \in S$, the quotient module $G/sG$ is projective $R/sR$-module.
\item [$(iii)$] $(\SSF, \SWC)$ is a complete cotorsion pair, i.e. $\SSF^{\perp}=\SWC$, $\SSF={}^{\perp}\SWC$ and for every $R$-module $M$, there are short exact sequences
    \[0 \rt C \rt G \rt M \rt 0  \ \ {\rm{and}} \ \ 0 \rt M \rt C' \rt G' \rt 0,\]
    such that $G, G' \in \SSF$ and $C, C' \in \SWC.$ This, in particular, implies that $\SSF$ is a precovering class and $\SWC$ is a preenveloping class.
\item [$(iv)$] $\SWC$ is an enveloping class \cite[Theorem 2.10(3)]{Tr1}.
\item [$(v)$] $\SSF$ is a covering class if and only if every flat $R$-module is $S$-strongly flat if and only if $R$ is an $S$-almost perfect ring \cite[Theorem 7.9]{BP}.
\end{itemize}
\end{facts}

\begin{remark}
The converse of Statement $(ii)$ of the previous facts is valid under certain conditions. Its general validity is a conjecture known as the Optimistic Conjecture  ($\OC$)  in \cite[1.1]{PS1}. The  ($\OC$) states the following: Let $S$ be a multiplicative subset of $R$ such that the projective dimension of $R_S$ as an $R$-module is less than or equal to $1$. Let $F$ be a flat R-module. Then, $F$ is $S$-strongly flat if $F_S$ is a projective $R_S$-module and if $G/sG$ is a projective $R/sR$-module for every $s$ in $S$.
\end{remark}

\s {\sc{Convention.}}\label{Conv-2}
In this paper, when we refer to $(\OC)$ holding for a pair $(R, S)$, we mean that $R$ is a commutative ring and $S \subseteq R$ is a multiplicative set that satisfies the conditions of the Optimistic Conjecture of \cite{PS1}. This means that the projective dimension of $R_S$ as an $R$-module does not exceed $1$, and a flat $R$-module $F$ is $S$-strongly flat if and only if $F_S$ is projective as an $R_S$-module and, for every $s\in S$, $F/sF$ is projective as an $R/sR$-module.

\subsection{Periodic strongly flat modules}\label{Subsec: PeriodicSF}
Let $\SX$ be an arbitrary class of modules. An $R$-module $M$ is called $\SX$-periodic, resp. pure $\SX$-periodic, if it fits into an exact sequence, resp. pure exact sequence, of the form $0\rt M\rt X\rt M\rt 0$, with $X\in\SX$. Using the terminology of \cite{BCE}, we say that $M$ is trivial if it belongs to $\SX$. In \cite[Theorem 2.5]{BG}  it is proved that every flat $\Prj$-periodic module is trivial or equivalently every pure $\Prj$-periodic module is trivial. Moreover, in \cite[Theorem 1.3]{Si} it is proved that every pure $\PPrj$-periodic module is trivial, where $\PPrj$ is the class of pure projective modules. In the following, we deal with the situations when every pure $\SSF$-periodic module is trivial.

\begin{proposition}
Let $R$ be a commutative ring and $S\subset R$  be a multiplicative subset such that $(\OC)$ holds for $(R, S)$. Let $M$ be an $R$-module fitting into a pure exact sequence $0\rt M\rt G\rt M\rt 0$ such that $G$ is $S$-strongly flat. Then $M$ is $S$-strongly flat.
\end{proposition}

\begin{proof}
Since $G$ is flat so is $M$. Therefore, $M_S$ and $M/sM$, for every $s\in S$, are flat. On the other hand, by \cite[Lemma 3.1]{BP} $G_S$ and $G/sG$ for every $s\in S$, are projective. Hence, by considering two exact sequences
\[0\rt M_S\rt G_S\rt M_S\rt 0,\]
\[0\rt M/sM\rt G/sG\rt M/sM\rt 0\]
and using \cite[Theorem 2.5]{BG}, we get that $M_S$ and $M/sM$, for every $s\in S$, are projective. Now it follows from $(\OC)$ that $M$ is $S$-strongly flat.
\end{proof}

\begin{corollary}
Let $R$ be a commutative ring and $S\subset R$  be a multiplicative subset such that $(\OC)$ holds. Let $0\rt M\rt G^1\rt G^2\rt \cdots \rt G^n\rt M\rt 0$ be a pure exact sequence such that for every $1\leq i\leq n$, $G^i$  is $S$-strongly flat. Then $M$ is $S$-strongly flat.
\end{corollary}

\begin{proof}
Since $G^i$, $1\leq i\leq n$, is $S$-strongly flat, it is flat, and hence $M$ is flat. Therefore $M_S$ and $M/sM$, for every $s\in S$, are flat. Moreover, $G^i_S$ and $G^i/sG^i$, for every $s\in S$ and $1\leq i\leq n$, are projective. Hence pure exact sequences
\[0\rt M_S\rt G^1_S\rt G^2_S\rt \cdots\rt G^n_S\rt M_S\rt 0,\] and
\[0\rt M/sM\rt G^1/sG^1\rt G^2/sG^2\rt \cdots \rt G^n/sG^n\rt M/sM\rt 0\]
in view of \cite[Corollary 1.4]{Si}, imply that $M_S$ and $M/sM$, for every $s\in S$, are projective. The result now follows from $(\OC)$.
\end{proof}

\subsection{Phantom morphisms}\label{Subsec: Phantom}
A morphism $\varphi: U\rt Z$ of $R$-modules is called a phantom morphism if the pullback along $\varphi$ of any short exact sequence with right term $Z$, as shown in the top row,
 \[\begin{tikzcd}
 &	0 \rar  &X\rar\dar[equals]&Y'\rar\dar&U\dar{\varphi}\rar& 0\\
 &	0\rar&X\rar &Y\rar& Z\rar & 0
	\end{tikzcd}\]
is a pure exact sequence, or equivalently, if the induced morphism
\[\Tor_1^R(X, \varphi): \Tor_1^R(X, U)\rt \Tor_1^R(X, Z)\]
of abelian groups is $0$ for every $R$-module $X$. A phantom morphism $\varphi: U\rt Z$ is trivial if it factors through a flat module.

\begin{remark}
Phantom morphisms are morphism versions of flat modules. For instance, one can see that $R$ is perfect if and only if every trivial phantom morphism is projective.
Recall that a morphism $f: X\rt A$ of $R$-modules is a projective morphism if $\Ext^1_R(f , B)=0$ for every $R$-module $B$. To illustrate this, let $R$ be a perfect ring and $\varphi: U\rt Z$ be a trivial phantom morphism. Let $F$ be a flat $R$-module such that $\varphi$ factors through it. By assumption $F$ is projective, so the pullback of any short exact sequence along morphism $F\rt Z$  is split. Therefore, the pullback along $\varphi$ is split. Hence $\varphi$ is projective. For the converse, let $F$ be a flat $R$-module. Then $1_F: F\rt F$ is a trivial phantom morphism. According to our assumption, it is a projective morphism, so $F$ is a projective $R$-module.
\end{remark}

The next proposition provides more evidence to support this point of view. It is a morphism version of Lemma 1.10 of \cite{PS2}. First, we need a lemma.

\begin{lemma}\label{IdealVersion4.1}
Let $R\rt R'$ be a morphism of commutative rings. Let $\varphi: U\rt Z$ be a morphism of $R$-modules such that $\Tor^R_i(R', U)=0=\Tor^R_i(R', Z)$, for $i\geq 1$. Then for an $R'$-module $N$ and $i\geq1$, $\Tor^R_i(N, \varphi)=0$ if and only if $\Tor^{R'}_i(N, R'\otimes_R \varphi)=0$.
\end{lemma}

\begin{proof}
Consider the following commutative diagram
\[\begin{tikzcd}
  	 &&\Tor^R_i(N, U) \ar{rrr}{\Tor^R_i(N, \varphi)}\dar{\cong}&&&\Tor^R_i(N, Z)\dar{\cong}\\
&&\Tor^{R'}_i(N, R'\otimes U)\ar{rrr}{\Tor^{R'}_i(N, R'\otimes\varphi)}&&& \Tor^{R'}_i(N, R'\otimes Z).
	\end{tikzcd}\]
Since by Lemma 4.1(b) of \cite{PS2}, the vertical morphisms are isomorphisms, we get the result.
\end{proof}

Recall that an R-module $M$ is said to be $S$-torsion if, for every element $x \in M$, there exists an element $s \in S$ such that $sx =0$. For any $R$-module $M$, the kernel of the canonical morphism $N \rt N_S$ is an example of an $S$-torsion module.

\begin{proposition}
Let $R$ be a commutative ring and $S\subseteq R$ be a multiplicative subset. Let $\varphi: U\rt Z$ be a morphism such that $\Tor^R_i(R/sR, U)=0=\Tor^R_i(R/sR, Z)$, for all $i\geq1$ and all $s\in S$. Then the morphism $\varphi$ is phantom if and only if the morphisms $\varphi/s\varphi$, for all $s \in S$, and $\varphi_S$ are phantom.
\end{proposition}

\begin{proof}
Let $\varphi$ be a phantom morphism. Let $s \in S$. We show that $\varphi/s\varphi:U/sU\rt Z/sZ$ is phantom. Let $0\rt X\rt Y\rt Z/sZ\rt 0$ be a short exact sequence of $R/sR$-modules. Let $c: Z \rt Z/sZ$ be the canonical projection. Since $\varphi$ is phantom, $c\varphi$ is phantom. Therefore $\delta$ in the pullback diagram
\[\begin{tikzcd}
 \delta: &	0 \rar  &X\rar\dar[equals]&W\rar\dar&U\dar{c\varphi}\rar& 0\\
\eta: &	0\rar&X\rar &Y\rar& Z/sZ\rar & 0
	\end{tikzcd}\]
is pure. Hence the following diagram
\[\begin{tikzcd}
 \delta/s\delta: &	0 \rar  &X\rar\dar[equals]&W/sW\rar\dar&U/sU\dar{\varphi/s\varphi}\rar& 0\\
\eta: &	0\rar&X\rar &Y\rar& Z/sZ\rar & 0
	\end{tikzcd}\]
is a pullback diagram and $\delta/s\delta$ is pure. Hence $\varphi/s\varphi$ is phantom. The fact that $\varphi_S$ is a phantom morphism follows by a similar argument.

For the converse, assume that $\varphi/s\varphi$, for all $s \in S$, and $\varphi_S$ are phantom morphisms. We show that $\Tor^R_1(N, \varphi)=0$, for every $R$-module $N$. To this end, we apply the same technique as in the proof of Lemma 1.10 of \cite{PS2}. Consider two short exact sequences
\[0\rt \Ker f_N\rt N\rt \im f_N\rt 0,\]
\[0\rt \im f_N\rt N_S\rt \Coker f_N\rt 0,\]
where $f_N:N\rt N_S$ is the canonical morphism. We have the following two induced exact sequences
\[ \Tor^R_1(\Ker f_N, \varphi)\rt \Tor^R_1(N,\varphi)\rt \Tor^R_1(\im f_N, \varphi),\]
\[ \Tor^R_2(\Coker f_N, \varphi)\rt \Tor^R_1(\im f_N,\varphi)\rt \Tor^R_1(N_S, \varphi)\]
We show that $\Tor^R_1(\Ker f_N, \varphi)=0$ and $\Tor^R_1(\im f_N, \varphi)=0$. To show $\Tor^R_1(\im f_N, \varphi)=0$, given the second long exact sequence, it is enough to show $\Tor^R_2(\Coker f_N, \varphi)=0$ and $\Tor^R_1(N_S,\varphi)=0$.

Now $N_S$ is an $R_S$-module and  by assumption $\varphi_S$ is phantom, so $\Tor^{R_S}_1(N_S, \varphi_S)=0$. Hence, by Lemma \ref{IdealVersion4.1}, we have $\Tor^R_1(N_S,\varphi)=0$.

On the other hand, $\Ker f_N$  and $\Coker f_N$ are $S$-torsion $R$-modules. By Lemma 4.3 of \cite{PS2}, for an $S$-torsion $R$-module $M$, $\Tor^R_i(M,\varphi)=0$, for any $i\geq 1$  if $\Tor^R_i(D,\varphi)=0$, for certain $R/sR$-modules $D$.  But by assumption,  $\Tor^{R/sR}_i(D,\varphi/s\varphi)=0$, for $i=1,2$. Hence by another use of Lemma \ref{IdealVersion4.1}, we get $\Tor^R_i(D,\varphi)=0$.
\end{proof}

\section{$S$-purity}\label{Sec: S-Purity}
Let $R$ be a commutative ring, and $S\subseteq R$ be a multiplicative subset that may contain some zero divisors, as in our setup.

\begin{definition}\label{Def; S-pure}
A short exact sequence $\eta: 0\rt N\rt M\rt L\rt 0$ of $R$-modules is called $S$-pure exact if it is pure exact, the induced sequence $\eta\otimes_R R_S$ is split exact and for every $s \in S$, the induced sequence $\eta\otimes_R R/sR$ is split exact.  In this case, $N$ is called an $S$-pure submodule of $M$. We let $\bS$  denote the collection of all $S$-pure exact sequences.
\end{definition}

\begin{example}
Every short exact sequence $\eta: 0\rt N\rt M\rt R_S\rt 0$ with right term $R_S$ is $S$-pure exact. Indeed, since $R_S$ is a flat $R$-module, $\eta$ is pure exact. Moreover, since $R_S$ is a projective $R_S$-module, $\eta_S: 0\rt M_S\rt N_S\rt R_S\rt 0$ splits. For the validity of the third condition note that $R_S\otimes_R R/sR=0$, for every $s\in S$.
\end{example}

\begin{lemma}\label{subfunctor}
The collection $\bS$ comprise an additive subfunctor $\Ext_\bS$ of $\Ext$.
\end{lemma}

\begin{proof}
It is plain that $\bS$ contains split short exact sequences and is closed under direct sums. We verify that $\bS$  is closed under pullbacks and pushouts.
Let $\eta: 0\rt Y\rt Z\rt X\rt 0$ be in $\bS$ and $\eta': 0\rt Y \rt U\rt A\rt 0$ be the pullback of $\eta$ along morphism $f:A\rt X$, as depicted in the following diagram
\[\begin{tikzcd}
\eta': &	0 \rar  &Y\rar\dar[equals]&U\rar\dar&A\dar{f}\rar& 0\\
\eta:&	0\rar & Y\rar& Z\rar & X \rar & 0.
	\end{tikzcd}\]
Since the class of pure exact sequences is closed under pullbacks, $\eta'$ is pure exact. Now apply the functors $-\otimes_R R_S$ and $-\otimes_R R/sR$ to the above pullback diagram. The induced short exact sequences $\eta_{S}=\eta \otimes_RR_S$ and $\eta_{R/sR}=\eta\otimes_RR/sR$ are split, because $\eta\in \bS$. Hence $\eta' \in \bS$. A similar argument, concludes that $\bS$ is closed under pushout.
\end{proof}

\begin{definition}
An $R$-module $M$ is $S$-pure injective if it is injective with respect to the class $\bS$. The collection of all $S$-pure injective $R$-modules will be denoted by $\SPI$.
\end{definition}

By definition, it is clear that every pure injective module is  $S$-pure injective. Moreover, the adjoint duality of Hom and tensor implies that $\Hom_R(R_S, X)$ and $\Hom_R(R/sR, X)$, for any $R$-module $X$, are $S$-pure injective.

\begin{remark}\label{SsI is Swc}
Every $S$-pure injective $R$-module $M$ is an $S$-weakly cotorsion $R$-module. Indeed, let $\eta: 0\rt M\rt N\rt R_S\rt 0 $ be an element in $\Ext^1_R(R_S, M)$. Since $\eta \in \bS$ and $M$ is  $S$-pure injective, then $\eta$ splits. Hence $\Ext^1_R(R_S, M)=0$ and $M$ is  $S$-weakly cotorsion.
\end{remark}

\begin{lemma}\label{EquivalentConditionsForS-StronglyPure}
A pure exact sequence $0\rt N\rt M\rt L\rt 0$ is $S$-pure if and only if the following two conditions are satisfied.
\begin{itemize}
  \item[$(i)$] For every $R_S$-module $X$, the induced sequence \[0\rt \Hom_R(L, X)\rt \Hom_R(M, X)\rt \Hom_R(N, X)\rt 0\] is exact.
  \item[$(ii)$] For every  $s\in S$ and  every $R/sR$-module $Y$, the induced sequence \[0\rt \Hom_R(L, Y)\rt \Hom_R(M, Y)\rt \Hom_R(N, Y)\rt 0\] is exact.
\end{itemize}
\end{lemma}

\begin{proof}
Let  $\eta: 0\rt N\rt M\rt L\rt 0$ be an $S$-pure exact sequence.
Since \[\eta_S: 0\rt N_S\rt M_S\rt L_S\rt 0\] is split, for every $R_S$-module $X$, the sequence
\[0\rt \Hom_{R_S}(L_S, X)\rt \Hom_{R_S}(M_S, X)\rt \Hom_{R_S}(N_S, X)\rt 0\]
is exact. Now by adjunction, we get the short exact sequence
\[0\rt \Hom_R(L, X)\rt \Hom_R(M, X)\rt \Hom_R(N, X)\rt 0.\] The validity of the Statement $(ii)$ follows similarly, using the split short exact sequence
\[\eta/s\eta: 0\rt N/sN\rt M/sM\rt L/sL\rt 0.\]

Now let $\eta$ be a pure exact sequence satisfying conditions $(i)$ and $(ii)$. We have to show that $\eta$ in $\bS$.
By $(i)$, the sequence
{\footnotesize{
\[0\rt \Hom_R(L, \Hom_{R_S}(R_S, N_S))\rt \Hom_R(M, \Hom_{R_S}(R_S, N_S))\rt \Hom_R(N, \Hom_{R_S}(R_S, N_S))\rt 0.\]}}
is exact. So by adjoint duality the sequence
\[0\rt \Hom_{R_S}(L_S, N_S)\rt \Hom_{R_S}(M_S, N_S)\rt \Hom_{R_{S}}(N_S, N_S)\rt 0\]
is exact. Hence the sequence $0\rt N_S\rt M_S\rt L_S\rt 0$ is split exact. By a similar argument, we can show that for every $s\in\bS$, the sequence
\[0\rt N/sN\rt M/sM\rt L/sL\rt 0\] is split exact.
\end{proof}

Let $N$ be an $R$-module and $\PE(N)$ denotes the pure injective envelope of $N$. Let $f_1: N \rt \PE(N)$, $f_2: N \rt \Hom_R(R_S, N_S)$ and, for every $s \in S$, $f_3^s: N \rt \Hom_R(R/sR, N/sN)$ be natural homomorphisms. In the next proposition, we apply these three maps to construct an $S$-pure injective preenvelope of $N$.

\begin{proposition}\label{SsI is PreEnveloping}
Let $N$ be an $R$-module. The map
\[\begin{tikzcd}[column sep=small]
 \eta:\ 0\ar{r} & N\ar{rrr}{[f_1\ f_2 \ f_3]^t} &&& {\PE}(N)\oplus \Hom_R(R_S, N_S)\oplus \prod_{s\in S}\Hom_R(R/sR, N/sN),
 \end{tikzcd}\]
provides an $S$-pure injective preenvelope of $N$.
\end{proposition}

\begin{proof}
Consider the exact sequence
 \[\begin{tikzcd}[column sep=small]
 \eta:\ 0\ar{r} & N\ar{rrr}{[f_1 \  f_2 \ f_3]^t} &&& {\PE}(N)\oplus \Hom_R(R_S, N_S)\oplus \prod_{s\in S}\Hom_R(R/sR, N/sN)\ar{r} & L\ar{r} & 0.
 \end{tikzcd}\]
Since $f_1:N \rt \PE(N)$ is a pure monomorphism, $\eta$ is pure exact. By Lemma \ref{EquivalentConditionsForS-StronglyPure}, to complete the proof, we need to verify that for every $R_S$-module $U$ and $R/sR$-module $W$, where $s \in S$, $\eta$ is $\Hom_R(-, U)$-exact and $\Hom_R(-, W)$-exact.

Let $\alpha\in\Hom_R(N,U)$. Since $U$ is an $R_S$-module, $\alpha: N\rt U$ can be extended to $\overline{\alpha}: N_S\rt U$. Now we define $\phi: \Hom_R(R_S, N_S)\rt U$, $\phi(g)=\overline{\alpha} g(1)$. It is easy to see that for every $x\in N$, $\phi f_2(x)=\alpha(x)$. Hence $\eta$ is $\Hom_R(-, U)$-exact.

To show that $\eta$ is $\Hom_R(-, W)$-exact, pick $\beta\in \Hom_R(N, W)$. Since for every $s\in S$, $W$ is an $R/sR$-module, $\beta: N\rt W$ can be extended to $\overline{\beta}: N/sN\rt W$. Now we define $\psi: \Hom_R(R/sR, N/sN)\rt W$, $\psi(h)=\overline{\beta}h(1+sR)$. Therefore for every $x\in N$, we have $\psi f_3(x)=\beta(x)$. Hence  $\eta$ is $\Hom_R(-, W)$-exact.
\end{proof}

Let $0\rt N \rt M\rt L\rt 0$ be an $S$-pure injective extension of module $N$. It is called a generator if for any other $S$-pure injective extension $0\rt N\rt M'\rt L'\rt 0$  of $N$, there exists a commutative diagram
\[\begin{tikzcd}
 &	0 \rar  &N\rar\dar[equals]&M'\rar\dar{g}&L'\dar{f}\rar& 0\\
&	0\rar & N\rar& M\rar & L\rar & 0.
	\end{tikzcd}\]
Moreover, such a generator is called minimal, if any commutative diagram
\[\begin{tikzcd}
 &	0 \rar  &N\rar\dar[equals]&M\rar\dar{1_M}&L\dar{f}\rar& 0\\
&	0\rar & N\rar& M\rar & L\rar & 0
	\end{tikzcd}\]
implies that $f$ is an automorphism.

The next proposition, by using a standard argument, shows that the class of $S$-pure injective modules is an enveloping class. We preface it with a lemma.

\begin{lemma}\label{SsI closed under Limit}
Let $I$ be a directed index set and $\lbrace 0\rt N\rt M_i\rt L_i\rt 0  \ | \   i,j\in I\rbrace$
be a direct system of the $S$-pure extensions of $N$. Then the sequence
\[0\rt N\rt \lim_{\rightarrow} M_i\rt \lim_{\rightarrow} L_i\rt 0\]
is an $S$-pure injective extension of $N$.
\end{lemma}

\begin{proof}
Since, for every $R$-module $A$, the functor $A\otimes_R-$ preserves direct limits and the direct limit is an exact functor, we get the result.
\end{proof}

\begin{proposition}\label{SsI is Enveloping}
Let $R$ be a commutative ring and $S\subseteq R$ be a multiplicative subset. Then every $R$-module $N$ has an $S$-pure injective envelope.
\end{proposition}

\begin{proof}
By Lemma \ref{SsI closed under Limit}, the class of $S$-pure extensions of $N$  is closed under direct limits. Hence by a similar argument as in the proof of Theorem 2.3.8 of \cite{Xu}, we get the result.
\end{proof}

\begin{remark}\label{LeftOrthogonalSsI is Flat}
By \cite[Lemma 3.4.1]{Xu}, for any ring $R$, ${}^{\perp}{(\PInj)} = \Flat$, see also \cite[Proposition 4]{H3}. Since $\PInj \subseteq \SPI$, ${}^{\perp}{(\SPI)}\subseteq {}^{\perp}{(\PInj)}=\Flat$. We use this fact to prove the next proposition.
\end{remark}

\begin{proposition}\label{LeftOrthogonalOfS-pureInjective}
Let $(R,S)$ be a pair satisfying $(\OC)$. Then $S\mbox{-}{\rm{SF}}={}^{\perp}{(\SPI)}$.
\end{proposition}

\begin{proof}
By Remark \ref{SsI is Swc}, we have $\SSF\subseteq {}^{\perp}{(\SPI)}$. The reverse inclusion can be proved using a similar argument as in Proposition 4 of \cite{H3}. For the reader's convenience, we present the proof. Let $M\in {}^{\perp}{(\SPI)}$. Consider the pushout diagram
  \[\begin{tikzcd}
\eta: &	0 \rar  &K\rar{i}\dar{e}&P\rar\dar&M\dar[equals]\rar& 0\\
\eta':&	0\rar & \SPE(K)\rar& U\rar & M \rar & 0
	\end{tikzcd}\]
where $e: K\rt \SPE(K)$ is an $S$-pure injective envelope, which exists by Proposition \ref{SsI is Enveloping}. Since $\eta'$ splits, there exists a morphism $g: P\rt \SPE(K)$ such that $gi=e$. Therefore we can construct the following pullback diagram.
\[\begin{tikzcd}
 &	0 \rar  &K\rar{i}\dar[equals]&P\rar\dar{g}&M\dar\rar& 0\\
&	0\rar&K\rar{e} &\SPE(K)\rar& Z\rar & 0
	\end{tikzcd}\]
Since the bottom row is in the class $\bS$ and $\bS$ is closed under pullbacks, the top row also belongs to $\bS$.
Hence the short exact sequences
\[0\rt K_S\rt P_S\rt M_S\rt 0 \ \ {\rm and} \ \ 0\rt K/sK\rt P/sP\rt  M/sM\rt 0,\]
for every $s\in S$, are split. So $M_S$ and $M/sM$ are projectives. Moreover, by Remark \ref{LeftOrthogonalSsI is Flat}, $M$ is flat. Hence the result follows since $(\OC)$ holds for $(R, S)$.
\end{proof}

\begin{proposition}\label{Proposition 17 of H1}
Let $(\OC)$ holds for the pair $(R,S)$. The following statements are equivalent.
\begin{itemize}
\item[$(1)$] The $S$-pure injective  envelope of every $S$-strongly flat module $G$ is $S$-strongly flat.
\item[$(2)$] The $S$-weakly cotorsion envelope every $S$-strongly flat module $G$ is $S$-pure injective.
\item[$(3)$] Every $S$-strongly flat and $S$-weakly cotorsion module is $S$-pure injective.
\end{itemize}
\end{proposition}

\begin{proof}
$(1)\Rightarrow (2).$ Let $G$ be an $S$-strongly flat module. Let $\eta: 0\rt G\rt \SPE(G)\rt L\rt 0$ be the $S$-pure injective envelope of $G$. Since $\eta$ in $\bS$, we have the following two split short exact sequences.
\[0\rt G_S\rt \SPE(G)_S\rt L_S\rt 0,\]
\[0\rt G/sG\rt \SPE(G)/s\SPE(G)\rt L/sL\rt 0.\]
By assumption $\SPE(G)$ is $S$-strongly flat. Hence, by \ref{Facts}, $\SPE(G)_S$ and $\SPE(G)/s\SPE(G)$, for every $s\in S$, are projectives. Therefore $L_S$ and $L/sL$ are projectives. By Optimistic Conjecture, we get $L$ is $S$-strongly flat. Since $\SPE(G)$ is $S$-weakly cotorsion and $L$ is $S$-strongly flat, by \cite[Proposition 2.1.4]{Xu}, we get $\eta$ is $S$-weakly cotorsion preenvelope. Therefore, since the $S$-weakly cotorsion envelope is a direct summand of $\SPE(G)$, it is $S$-pure injective.

$(2)\Rightarrow (1).$ Let $G$ be an $S$-strongly flat module and $\varphi: G\rt C$ be the $S$-weakly cotorsion envelope of $G$. By Theorem 3.4.2 of \cite{Xu},  $\Ext^1_R(\Coker\varphi, \SWC)=0$. Hence $\Coker\varphi$ is $S$-strongly flat. Therefore $C$ is $S$-strongly flat.

$(2)\Leftrightarrow (3).$ Follow immediately.
\end{proof}

\subsection{Closedness under extension}\label{Subsec: ClosednessUnder Ext}
In \cite[\S 3.3.5]{Xu} the situation is studied where the class of pure injective modules is closed under extensions. It is shown that this is the case if and only if every cotorsion module is pure injective. Moreover, in \cite{HR} the authors characterize rings over which every cotorsion module is pure injective.
In this subsection, we show that the class of $S$-pure injective modules is closed under extensions if and only if every $S$-weakly cotorsion module is $S$-pure injective.

\begin{lemma}\label{Lemma3.5.2 of Xu}
Let the class of $S$-pure injective modules be closed under extensions. Consider the commutative diagram
\[\begin{tikzcd}
 	 &  & M\rar{i}\dar{p}& \SPE(M)\dlar[dotted][near start]{u}\dar{q} \\
 	0\rar&X\rar &Y\rar{\beta}& Z\rar & 0
	\end{tikzcd}\]
where $X$ is an $S$-pure injective module and $\SPE(M)$ is the $S$-pure injective envelope of $M$. Then, there exists a morphism $u: \SPE(M)\rt Y$ such that $ui=p$ and $\beta u=q$.
\end{lemma}

\begin{proof}
  Consider the following pullback diagram
   \[\begin{tikzcd}
 	0 \rar  &X\rar{h}\dar[equals]&W\rar{v}\dar{w}&\SPE(M)\dar{q}\rar& 0\\
	0\rar&X\rar &Y\rar& Z\rar & 0.
	\end{tikzcd}\]
  Since $X$ and $\SPE(M)$ are $S$-pure injective, so by the assumption we get $W$ is also $S$-pure injective. By the property  of the pullback diagram, there exists a unique morphism
   $\phi: M\rt W$ such that $i=v\phi$ and $p= w \phi$. Now since $W$ is $S$-pure injective, by the definition, there exists a morphism $\psi: \SPE(M)\rt W$ such that $\phi=\psi i$.
   Therefore $i=v\phi=v\psi i$. Since $i$ is the $S$-pure injective envelope of $M$, $v\psi$ is an automorphism. Set $u=w\psi(v\psi)^{-1}$. It's easy to verify that $\beta u=q$ and $ui=p$.
\end{proof}

\begin{lemma}\label{Lemma 3.5.3 of Xu}
Let the class of $S$-pure injective modules be closed under extensions. Let $M$ be an $R$-module and set $F:=\SPE(M)/M$. Consider the diagram
 \[\begin{tikzcd}
 	& & &F\dar{g}\dlar[dotted]&\\
	0\rar&X\rar &Y\rar{\beta}& Z\rar & 0,
	\end{tikzcd}\]
where $X$ is an $S$-pure injective module. Then there exists morphism $u':F\rt Y$ such that $\beta u'=g$.
\end{lemma}

\begin{proof}
Consider the following diagram
\[\begin{tikzcd}
 	 & 0\rar & M\rar{i}\dar{0}& \SPE(M)\dlar[dotted][near start]{u}\dar{gj}\rar{j}& F\dlar{g}\rar&0 \\
 	0\rar&X\rar &Y\rar{\beta}& Z\rar & 0.
	\end{tikzcd}\]
By  Lemma \ref{Lemma3.5.2 of Xu}, there exists morphism $u:\SPE(M)\rt Y$ such that $\beta  u=gj$. Now we define $u': F\rt Y$, $d\longmapsto u(x)$ where $j(x)=d$. It is clear that $g=\beta u'$.
\end{proof}

\begin{lemma}\label{Lemma 3.5.4 of Xu}
Let the class of $S$-pure injective modules be closed under extensions. Let $M$ be an $R$-module and set $F:=\SPE(M)/M$. Then $\Ext^1_R(F, X)=0$, for any $S$-pure injective module $X$. Moreover, if $(\OC)$ holds for the pair $(R,S)$, then $F$ is $S$-strongly flat.
\end{lemma}

\begin{proof}
Let $\eta: 0\rt X\rt Y\rt F\rt 0$ be an extension in $\Ext^1_R(F, X)$. Consider the following diagram
\[\begin{tikzcd}
 	& & &F\dar{1_F}\dlar[dotted][near start]{u'}&\\
	0\rar&X\rar &Y\rar{\beta}& F\rar & 0.
	\end{tikzcd}\]
By Lemma \ref{Lemma 3.5.3 of Xu}, there exists $u': F\rt Y$ such that $\beta u'=1_F$. Hence the extension $\eta$ is split.  If $(\OC)$ holds for the pair $(R,S)$,  Proposition \ref{LeftOrthogonalOfS-pureInjective} implies that  $F$ is $S$-strongly flat.
\end{proof}

The following theorem is inspired by a well-known classification of Xu rings, as shown, for example, in \cite[Theorem 3.5.1]{Xu}.

\begin{theorem}\label{Theorem 3.5.1 of Xu}Let $(\OC)$ holds for the pair $(R,S)$. The following statements are equivalent.
\begin{itemize}
\item[$(1)$] The class of $S$-pure injective modules is closed under extensions.
\item[$(2)$] For any $R$-module $M$, $\SPE(M)/M$ is $S$-strongly flat.
\item[$(3)$] Every $S$-weakly cotorsion module is $S$-pure injective.
\end{itemize}
\end{theorem}

\begin{proof}
$(1)\Rightarrow (2).$ Follows by the previous three lemmas.

$(2)\Rightarrow (3).$ Let $C$ be an $S$-weakly cotorsion module and   \[\eta: \ 0\rt C\rt \SPE(C)\rt \SPE(C)/C\rt 0\] be the $S$-pure injective envelope of $C$. Since $\SPE(C)/C$ is $S$-strongly flat, the short exact sequence $\eta$ is split exact. Hence $C$ is $S$-pure injective.

$(3)\Rightarrow (1).$ Let $0\rt X'\rt X\rt X''\rt 0$ be a short exact sequence with $X'$ and $X''$, $S$-pure injective. We  show that $X$ is $S$-pure injective. By Remark \ref{SsI is Swc}, $X'$ and $X''$ are $S$-weakly cotorsion $R$-modules and the class of $S$-weakly cotorsion $R$-modules is closed under extensions, we get $X$ is $S$-weakly cotorsion. Hence, by assumption, $X$ is $S$-pure injective.
\end{proof}

\section{$S$-phantom morphisms}\label{Sec: S-Phantom Morphisms}
Let $R$ be a commutative ring and $S\subseteq R$  be a multiplicative subset that may contain some zero-divisors.

\begin{definition}\label{Def: S-Phantom}
A morphism $f: G\to Z$ in $\Mod R$ is an $S$-phantom morphism if the pullback along $f$ of any short exact sequence with right term $Z$, as in the diagram
\[\begin{tikzcd}
\eta': &	0 \rar  &X\rar\dar[equals]&U\rar\dar& G \dar{f}\rar& 0\\
\eta:&	0\rar & X\rar& Y\rar & Z \rar & 0
	\end{tikzcd}\]
belongs to $\bS$.
\end{definition}

\begin{example}
Let $G$ be an $S$-strongly flat module. Since $G$ is flat and $G_S$ and $G/sG$ are projective modules, it follows that every morphism $f: G\rt X$ is $S$-phantom. Moreover,  Lemma \ref{subfunctor} implies that every morphism $f: X\rt G$, is an $S$-phantom morphism.
\end{example}

\begin{remark}
By Lemma \ref{subfunctor}, we observe that a morphism $f: G\rt Z$ is an $S$-phantom morphism if, for every $R$-module $X$, the morphism
\[\Ext^1_R(f, X): \Ext^1_R(Z, X)\rt \Ext^1_R(G, X)\]
takes value in the subgroup $\Ext^1_{\bS}(G, X)$.
\end{remark}

Recall that a two sided ideal $\SI$ of an additive category $\SA$ is an additive subfunctor of the additive bifunctor $\Hom: \SA^{\op}\times \SA\rt Ab$  that associates to every pair $A$ and $B$ of objects in $\SA$ a subgroup $\SI(A, B)\subseteq \Hom_\SA(A, B)$ such that for all $f \in \SI(A, B)$, $g\in \Hom_\SA(B, D)$ and $h\in \Hom_\SA(C, A)$, $gfh\in\SI(C, D)$. Let $\SX$ be an additive subcategory of the category of $R$-modules. Let $\CI(\SX)$ be the ideal generated by the identity morphisms $1_X$, for every $X\in\SX$. This is the ideal of morphisms in the category of $R$-modules that factor through an object in $\SX$.

\begin{remark}
It is easy to see that the collection of $S$-phantom morphisms forms an ideal. The ideal of all $S$-phantom morphisms will be denoted by $\SPh$. An ideal is called $S$-phantom if it consists of $S$-phantom morphisms. So $\SPh$ is an $S$-phantom ideal.
\end{remark}

\begin{example}
Let $\SX$ be a class of $S$-strongly flat modules. Then $\CI(\SX)$ is an $S$-phantom ideal.
\end{example}

To verify that a morphism $f: F\rt X$ is $S$-phantom, it suffices to show that the pullback along $f$ of a short exact sequence of the form $0\rt K\rt P\rt X\rt 0$ belongs to $\bS$, where $P$ is a projective module. We show this in the following lemma.

\begin{lemma}\label{ProjectiveRemark}
A morphism $f: F\rt X$ is $S$-phantom if and only if the pullback along $f$ of a short exact sequence of the form $0\rt K\rt P\rt X\rt 0$, with $P$ a projective $R$-module, belongs to $\bS$.
\end{lemma}

\begin{proof}
Let $0\rt Y\rt Z\rt X\rt 0$ be an arbitrary short exact sequence of modules. Since $P$ is a projective module, there is a morphism $P\rt Z$ and therefore a morphism $g: K\rt Y$. Now consider the following commutative diagram.
\begin{equation*}
\begin{tikzcd}[column sep=small]
 &\delta: &0\ar{rr}&& K\dlar{g}\ar{rr}\ar[equals]{dd} &&U'\dlar\ar{rr}\ar{dd} &&F\dlar[equals]\ar{dd}[near start]{f}\ar{rr}&&0\\
 \gamma: &  0\ar{rr}&&   Y\ar{rr}\ar[equals]{dd} &&U\ar{rr}\ar{dd}&&F\ar{dd}[near start]{f}\ar{rr}&&0&& \\
    &\eta: ~~~~~~~~~&0\ar{rr}&&K\dlar{g}\ar{rr}&&P\ar{rr}\dlar&&X\dlar[equals]\ar{rr}&&0\\
    & 0\ar{rr}&&Y\ar{rr}&&Z\ar{rr}&&X\ar{rr}&&0
\end{tikzcd}
\end{equation*}
We have to show that $\gamma$ is in $\bS$. Since by the assumption $\delta$ is the pullback of $S$-phantom morphism $f$ along $\eta$, so in $\bS$. Now since $\bS$ is closed under pushout, $\gamma$ that is the pushout of $\delta$, is in $\bS$.
\end{proof}

\begin{proposition}\label{S-phantomIsOrthogonalOfSpureinj}
A morphism $f: F\rt X$ is $S$-phantom if and only if for every $S$-pure injective module $Y$,  $\Ext^1_R(f, Y)=0$. In other words $\SPh={}^{\perp}(\SPI)$.
\end{proposition}

\begin{proof}
Let $f: F\rt X$ be an $S$-phantom morphism and $Y$ be an $S$-pure injective module. Then the induced morphism
  \[\Ext^1_R(f, Y): \Ext^1_R(X, Y)\rt \Ext^1_R(F, Y)\]
is zero. Indeed,  since $f$ is $S$-phantom, the pullback of any extension in $\Ext^1_R(X, Y)$ along $f$ is in $\bS$ and start with $S$-pure injective module $Y$. Hence it is split.

  For the converse, let the morphism $f: F\rt X$ satisfies $\Ext^1_R(f, Y)=0$ for every $S$-pure injective module $Y$. To show that $f$ is $S$-phantom, by Lemma \ref{ProjectiveRemark}, it is enough to show that the pullback along $f$ of a short exact sequence of the form $0\rt K\rt P\rt X\rt 0$, where $P$ is projective, is in $\bS$. Let $e: K\rt \SPE(K)$ be the $S$-pure injective envelope of $K$ and consider the following commutative diagram
  \begin{equation*}
\begin{tikzcd}[column sep=small]
 &\delta: &0\ar{rr}&& K\dlar{e}\ar{rr}{i}\ar[equals]{dd} &&U'\dlar\ar{rr}\ar{dd} &&F\dlar[equals]\ar{dd}[near start]{f}\ar{rr}&&0\\
 \gamma: &  0\ar{rr}&&   \SPE(K)\ar{rr}\ar[equals]{dd} &&U\ar{rr}\ar{dd}&&F\ar{dd}[near start]{f}\ar{rr}&&0&& \\
    &\eta: ~~~~~~~~~&0\ar{rr}&&K\dlar{e}\ar{rr}&&P\ar{rr}\dlar&&X\dlar[equals]\ar{rr}&&0\\
    & 0\ar{rr}&&\SPE(K)\ar{rr}&&Z\ar{rr}&&X\ar{rr}&&0.
\end{tikzcd}
\end{equation*}
We have to show that $\delta$ is in $\bS$. By the assumption, the short exact sequence $\gamma$ is split, so there exists a morphism $g: U'\rt \SPE(K)$ such that $gi=e$. Now the proof can be completed by a similar argument as at the end of the proof of Proposition \ref{LeftOrthogonalOfS-pureInjective}, which means $\delta$ is in $\bS$.
\end{proof}

Recall that a pair $(\CI, \CJ)$ of ideals of the category of $R$-modules is an ideal cotorsion pair \cite[Definition 12]{FGHT} if $\CJ=\CI^{\perp}$ and $\CI={}^{\perp} \CJ$, where
\[\CI^{\perp}=\{ j| ~\Ext^1_R(i, j)=0, ~~~{\rm for~ all} ~i\in\CI \},\]
\[{}^{\perp} \CJ =\{i| ~\Ext^1_R(i, j)=0, ~~~{\rm for~ all} ~j\in\CJ \}.\]

\begin{remark}
The above proposition may be reformulated as the ideal analog of Proposition \ref{LeftOrthogonalOfS-pureInjective}, that is,
\[{}^\perp (\CI(\SPI))=\SPh.\]
This, in particular, implies that $(\SPh, (\SPh)^{\perp})$ is an ideal cotorsion pair.
\end{remark}

\s Let $\CI$ be an ideal in the category of $R$-modules. An $\CI$-precover of an $R$-module $M$ is a morphism $i: X\rt M$ in $\CI$  such that any other morphism $i': X'\rt M$ in $\CI$ factors through $i$.
An $\CI$-precover $i: X\rt M$ is special if it is obtained as the pushout of a short exact sequence $\eta$ along a morphism $j: Z\rt N$ in $\CI^{\perp}$:

\[\begin{tikzcd}
\eta: &	0 \rar  &Z\rar\dar{j}&Y\rar\dar&M\dar[equals]\rar& 0\\
&	0\rar&N\rar &X\rar{i}& M\rar & 0.
	\end{tikzcd}\]

The ideal $\CI$ is special precovering if every $R$-module $M$ has a special $\CI$-precover $i:X\rt M$. A $\CJ$-preenvelope and a special $\CJ$-preenvelope for an ideal $\CJ$ is defined dually.

\begin{theorem}\label{Th: S-Phantom Precover}
Let $R$ be a commutative ring and $S\subseteq R$ be a multiplicative subset. Then every $R$-module has a special $S$-phantom precover.
\end{theorem}

\begin{proof}
  Let $M$ be an $R$-module and consider the short exact sequence $\eta: 0\rt K\rt P\rt M\rt 0$. By taking the pushout along the $S$-pure injective envelope $e: K\rt \SPE(K)$, we get the diagram.
   \[\begin{tikzcd}
 &	0 \rar  &K\rar{i}\dar{e}&P\rar\dar&M\dar[equals]\rar& 0\\
&	0\rar &\SPE(K)\rar& F\rar{f} & M \rar & 0.
	\end{tikzcd}\]
  We claim that $f: F\rt M$ is an $S$-phantom morphism. To show this, by Lemma \ref{ProjectiveRemark}, it is enough to show that the pullback along $f$ of the short exact sequence $\eta$ is in $\bS$, i.e. the short exact sequence $\delta$ in the diagram
  \[\begin{tikzcd}
& \delta: \ \	0 \rar  &K\rar{i}\dar[equals]&F'\rar\dar&F\dar{f}\rar& 0\\
&	0\rar&K\rar{i} &P\rar& M\rar & 0,
	\end{tikzcd}\]
is in $\bS$. By composing the two above diagrams, we get the following diagram.
   \[\begin{tikzcd}
 &	0 \rar  &K\rar{i}\dar{e}&F'\rar\dar\dlar[dotted]{g}&F\dar{f}\rar\dlar{1_F} & 0\\
&	0\rar &\SPE(K)\rar& F\rar{f} & M \rar & 0.
	\end{tikzcd}\]
Now by the homotopy, there exists a morphism $g: F'\rt \SPE(K)$ such that $e=gi$. By the similar argument as at the end of the proof of Proposition \ref{LeftOrthogonalOfS-pureInjective}, we get $\delta$ in $\bS$.

To complete the proof, we note that since $\Ext^1_R(\SPh,e)=0$, so $e\in (\SPh)^\bot$. Hence  by \cite[Proposition 8]{H3}, $f$ is a special phantom precover.
\end{proof}

An ideal cotorsion pair $(\CI, \CJ)$  is called special precovering, resp. special preenveloping, if the ideal $\CI$, resp.  $\CJ$, is special precovering, resp. special preenveloping.

\begin{remark} The above theorem implies that the ideal cotorsion pair $(\SPh, (\SPh)^{\perp})$ is special precovering and Salce's Lemma \cite[Theorem 18]{FGHT}, implies that $(\SPh)^{\perp}$ is special preenveloping.  In this case, we say that ideal cotorsion pair $(\SPh, (\SPh)^{\perp})$ is complete.  Moreover, by Salce's Lemma, we may characterize the ideal $(\SPh)^{\perp}$ as the object ideal generated by modules $F'$ that arise in the short exact sequence
\[0\rt Y\rt F'\rt E\rt 0\]
where $Y$ is $S$-pure injective and $E$ is an injective modules.
To see this characterization, let $M$ be an $R$-module. Let $0\rt M\rt E\rt X\rt 0$ be injective envelope of $M$. Since $\SPh$ is special precovering, for $R$-module $X$ there exists a special phantom precover
 \[\begin{tikzcd}
 &	0 \rar  &Y\rar&F\rar{f}& X \rar& 0
	\end{tikzcd}\]
where $Y$ is $S$-pure injective. Consider the following pullback diagram
\[\begin{tikzcd}
 & & & 0\dar &0\dar\\
 & & & Y\rar[equals]\dar& Y\dar\\
 &	0 \rar  &M\rar{j}\dar[equals]&F'\rar\dar&F\dar{f}\rar & 0\\
&	0\rar &M\rar& E\rar \dar& X \rar\dar & 0\\
&  & & 0& 0.
	\end{tikzcd}\]
By the argument, as in the proof of Salce's Lemma, we obtain that $j: M\rt F'$ is special $(\SPh)^{\perp}$-preenvelope of $M$ such that it is an extension of $Y$ by an injective module $E$. Since $F'$ is an object in $(\SPh)^{\perp}$, it follows that every morphism in $(\SPh)^{\perp}$ factors through an object in $(\SPh)^{\perp}$. Hence $(\SPh)^{\perp}$ is an object ideal.
 \end{remark}

We end this section by the following observation regarding ideal $\SPh$  which follows by  \cite[Theorem 1]{FGHT}. First, we recall some notions and notations. Let $(\CA,\varepsilon )$ be an exact category and  $\CF$ be a subfunctor of $\Ext$. The $\CF$-phantom morphisms form an ideal \cite{FGHT}, denoted by $\Phi(\CF)$. An additive subfunctor $\CF\subseteq \Ext$ has enough injective morphism if for every object $M\in\CA$ there is an $\CF$-conflation
\[\begin{tikzcd}
 \eta: & 0\rar& M\rar{e}& N\rar&L\rar& 0,
\end{tikzcd}\]
 where $e$ is an $\CF$-injective morphism. We say that the subfunctor $\CF\subseteq\Ext$ has enough special injective morphisms if for every object $M\in\CA$, there exists an $\CF$-conflation $\eta$ as above where $e$ obtained as the pullback along an $\CF$-phantom morphism.

\begin{corollary}
Let $\SPh$ be the ideal of $S$-phantom morphisms. Then $\Ext_\bS\subseteq \Ext$ has enough injective morphisms and $\SPh=\Phi(\Ext_\bS)$.
\end{corollary}

\begin{proof}
Since $\SPh$ is a special precovering ideal, the result follows from Theorem 1 of \cite{FGHT}.
\end{proof}

\subsection{Existence of $S$-phantom cover}\label{Subsec: S-Phantom Cover}
In this subsection, we investigate the situations in which the $S$-phantom ideal is a covering ideal. As before, let $R$ be a commutative ring and $S\subseteq R$ be a multiplicative subset.

\begin{proposition}\label{EveryPhantIsSphanImpliesS-almostPerfect}
If every phantom morphism is $S$-phantom, then the ring $R$ is $S$-almost perfect.
\end{proposition}

\begin{proof}
Let $F$ be a flat $R_S$-module. Since $F$ is also a flat $R$-module, then $1_F$ is a phantom morphism. By assumption $1_F$ is $S$-phantom. Hence $F_S$ and $F/sF$, for every $s\in S$, are projective. Now since $F\cong F\otimes_R R_S\cong F_S$, we get $F$ is a projective $R$-module.
Therefore $F$ is an  $S$-strongly flat $R$-module. So, \cite[Lemma 3.1]{BP} implies that  $F\cong F\otimes_R R_S$  is a projective $R_S$-module. Hence $R_S$ is a perfect ring.

Now we show that $R/sR$, for every $s\in S$, is perfect. To show this, as in the proof of Lemma 7.8 of \cite{BP}, it's enough to show that every Bass flat $R/sR$-module is projective.  Let $s$ be an element in $S$, so for the Bass flat $R/sR$-module $\bar{B}$ we have isomorphism $\bar{B}\cong B/sB$ where $B$ is a Bass flat $R$-module. We consider the phantom morphism $1_B$. By the assumption $1_B$ is $S$-phantom. So $B_S$ and $B/sB$ are projective. Hence $\bar{B}$ is projective.
\end{proof}

In the following, we show that the converse of the above proposition holds for trivial phantoms.

\begin{proposition}
Let $R$ be an $S$-almost perfect ring. Then  every trivial phantom is $S$-phantom.
\end{proposition}

\begin{proof}
Let $\varphi: U\rt Z$ be a trivial phantom morphism. Let $F$ be a flat module such that $\varphi$ factors through it. Since  by the assumption, $F$ is $S$-strongly flat, so the  morphism  $F\rt Z$ is $S$-phantom. Hence, the morphism $\varphi$ is $S$-phantom.
\end{proof}

\begin{proposition}\label{S-almostperfectImpliesS-phantoms}
Let $(\OC)$ holds for $(R, S)$. If $R$ is an $S$-almost perfect ring, then every phantom is $S$-phantom.
\end{proposition}

\begin{proof}
Let $\varphi$ be a phantom morphism. In order to show that $\varphi$ is $S$-phantom, by Proposition \ref{S-phantomIsOrthogonalOfSpureinj}, it is enough to show that $\Ext^1_R(\varphi, Y)=0$, for every $S$-pure injective $R$-module $Y$.  Consider the following equalities.
\[{}^{\perp} (\SPI)=\SSF=\Flat={}^{\perp}(\PInj),\]
where $\PInj$ is the class of pure injective $R$-modules. The first equality follows by Proposition \ref{LeftOrthogonalOfS-pureInjective}. The second equality follows since $R$ is $S$-almost perfect and the third, follows by \cite[Lemma 3.4.1]{Xu}.
Since $\varphi$ is phantom, by \cite[Proposition 6]{H1}, $\Ext^1_R(\varphi, X)=0$ for every pure injective $R$-module $X$. Now the result follows by the equality ${}^{\perp} (\SPI)={}^{\perp}(\PInj)$.
\end{proof}

\begin{proposition}
Let $(\OC)$ holds for $(R, S)$. If $R$ is an $S$-almost perfect ring, then the ideal of $S$-phantom morphisms is covering.
\end{proposition}

\begin{proof}
Since $R$ is $S$-almost perfect, by Proposition \ref{S-almostperfectImpliesS-phantoms}, every phantom morphism is $S$-phantom. Hence the result follows from Theorem 7 of \cite{H1}.
\end{proof}

An $S$-phantom morphism $\varphi: U\rt Z$ is $S$-trivial if it factors through an $S$-strongly flat module. By the above results, we have the following corollary, compare \cite[Theorem 10]{H1}.

\begin{corollary}\label{Theorem 10 of H1}
   For an $R$-module $N$, the following conditions are equivalent.
   \begin{itemize}
     \item[$(1)$] Every $S$-phantom morphism $g:M\rt N$ is $S$-trivial.
     \item[$(2)$] The $S$-phantom cover of $N$ is $S$-strongly flat.
     \item[$(3)$] The $S$-phantom cover of $N$ is the $S$-strongly flat cover.
     \item[$(4)$] The kernel of the $S$-strongly flat cover of $N$ is $S$-pure injective.
   \end{itemize}
\end{corollary}

\section{Optimistic Conjecture}\label{Sec: OC}
In this section we prove an ideal version of the Optimistic Conjecture raised in \cite{PS1}. Let $R$ be a commutative ring and $S\subseteq R$ be a multiplicative subset.

\begin{proposition}\label{Prop: OC1}
Let $\varphi:U\rt Z$ be an $S$-phantom morphism. The following hold.
\begin{itemize}
  \item [$(1)$] $\varphi_S: U_S\rt Z_S$ is a projective morphism.
  \item [$(2)$]  For every $s\in S$, $\varphi/s\varphi: U/sU\rt Z/sZ$ is a projective morphism.
\end{itemize}
 \end{proposition}

\begin{proof}
$(1).$ Let $\eta: 0\rt M\rt N\rt Z_S\rt 0$ be an extension of $R_S$-modules. Consider the pullback diagram
\[\begin{tikzcd}
 \delta: &	0 \rar  &M\rar\dar[equals]&L\rar\dar&U_S\dar{\varphi_S}\rar& 0\\
\eta: &	0\rar&M\rar &N\rar& Z_S\rar & 0
	\end{tikzcd}\]
of $R_S$-modules. We show that $\delta$ is split exact. Let $f: Z\rt Z_S$ be the localization map and consider the pullback diagram
\[\begin{tikzcd}
 \varepsilon &	0 \rar  &M\rar\dar[equals]&W\rar\dar&U\dar{f\varphi}\rar& 0\\
\eta: &	0\rar&M\rar &N\rar& Z_S\rar & 0
	\end{tikzcd}\]
of $R$-modules. Since the $S$-phantom morphisms form an ideal, $f\varphi$ is an $S$-phantom morphism. So $\varepsilon$ is pure exact and the short exact sequence
$\varepsilon_S:\ \ 0\rt M\rt W_S\rt U_S\rt 0$ is split. But $\varepsilon_S\cong \delta$. So $\delta$ is split.

$(2).$ Let $s\in S$. Let $\eta: 0\rt M\rt N\rt Z/sZ\rt 0$ be an extension of $R/sR$-modules. Consider the pullback diagram
\[\begin{tikzcd}
 \delta: &	0 \rar  &M\rar\dar[equals]&L\rar\dar&U/sU\dar{\varphi/s\varphi}\rar& 0\\
\eta: &	0\rar&M\rar &N\rar& Z/sZ\rar & 0
	\end{tikzcd}\]
of $R/sR$-modules. We show that $\delta$  is split. Let $\pi: Z\rt Z/sZ$ be the canonical projection. Then $\pi\varphi$ is an $S$-phantom morphism. Consider the pullback diagram
\[\begin{tikzcd}
\varepsilon: &	0 \rar  &M\rar\dar[equals]&W\rar\dar&U\dar{\pi\varphi}\rar& 0\\
\eta: &	0\rar&M\rar &N\rar& Z/sZ\rar & 0
	\end{tikzcd}\]
of $R$-modules. Since $\pi\varphi$ is an $S$-phantom morphism, $\varepsilon$ is pure exact and the exact sequence
$\varepsilon/s\varepsilon: \ \  0\rt M\rt W/sW\rt U/sU\rt 0$ of $R/sR$-modules is split. But $\varepsilon/s\varepsilon\cong\delta$. Hence $\delta$ is split exact.
\end{proof}

\begin{definition}
Let $R$ be a commutative ring and $S\subseteq R$ be a multiplicative subset. An $R$-module $M$ is called $S$-almost flat, if $\Tor_1^R(R/sR, M)=0$, for every $s\in S$.
\end{definition}

In the following, we demonstrate that the converse of the previous proposition is true when we examine phantom morphisms with codomains that are $S$-almost flat modules. For example, if the codomain of a phantom map is a flat module, then we have the converse of the previous proposition.

\begin{proposition}\label{Prop: OC2}
Let $\varphi: U\rightarrow Z$ be a phantom morphism, where $Z$ is $S$-almost flat. If $\varphi_S: U_S\rightarrow Z_S$ is a projective morphism, and for every element $s \in S$, $\varphi/s\varphi: U/sU\rightarrow Z/sZ$ is a projective morphism, then $\varphi$ is $S$-phantom.
\end{proposition}

\begin{proof}
Let $\varphi$ be phantom with flat $R$-module $Z$. Assume that $\varphi_S$ and $\varphi/s\varphi$, for every $s\in S$, are projective morphisms. Let $\eta: 0\rt X\rt Y\rt Z\rt 0$ be a short exact sequence of $R$-modules and $\delta: 0\rt X\rt W\rt U\rt 0$ be the short exact sequence obtained from the pullback of $\eta$ along $\varphi$. We need to show that $\delta\in\bS$. First, we note that since $\Tor_1^R(R/sR, Z)=0$, $\eta$ is a pure exact sequence and since $\varphi$ is phantom then $\delta$ is a pure exact sequence.

In order to show that the induced short exact sequence $\delta_S: 0\rt X_S\rt W_S\rt U_S\rt 0$ is split, note that $\delta_S$ is obtained from the pullback diagram
 \[\begin{tikzcd}
 \delta_S &	0 \rar  &X_S\rar\dar[equals]&W_S\rar\dar&U_S\dar{\varphi_S}\rar& 0\\
\eta_S: &	0\rar&X_S\rar &Y_S\rar& Z_S\rar & 0
	\end{tikzcd}\]
of $R_S$-modules.  By assumption $\varphi_S$ is a projective morphism, so $\delta_S$ is split.

 To complete the proof, we need to show that the sequence $0\rt X/sX\rt W/sW\rt U/sU\rt 0$ is split. Consider the pullback diagram
\[\begin{tikzcd}
 \delta/s\delta &	0 \rar  &X/sX\rar\dar[equals]&W/sW\rar\dar&U/sU\dar{\varphi/s\varphi}\rar& 0\\
\eta/S\eta: &	0\rar&X/sX\rar &Y/sY\rar& Z/sZ\rar & 0
\end{tikzcd}\]
 of $R/sR$-modules. Since $\varphi/s\varphi$ is a projective morphism, $\delta/s\delta$ is split.
\end{proof}

\section*{Acknowledgments}
The research of the first and the last authors is supported by the National Natural Science Foundation of China (Grant No. 12101316) and is partially supported by the Belt and Road Innovative Talents Exchange Foreign Experts project (Grant No. DL2023014002L). The second author's work is based on research funded by Iran National Science Foundation (INSF) under project No. 4001480. The second author also thanks the Nanjing University of Information Science and Technology (NUIST) for its warm hospitality and support during a visit by him. The research of the third author is supported by a grant from IPM. Part of this work was done while the second and third authors were visiting the Institut des Hautes \'{E}tudes Scientifiques (IHES) in Paris, France. They express their gratitude for the support and wonderful atmosphere at IHES.

\end{document}